\documentclass[11pt]{amsart}

\usepackage{amsmath,amssymb}

\usepackage{mathtools}

\usepackage[all,cmtip]{xy}

\allowdisplaybreaks

\begin{document}

% useful shortcuts
\newtheorem{thm}{Theorem}
\newtheorem{lem}{Lemma}
\newtheorem{prop}{Proposition}
\newtheorem{cor}{Corollary}
\newtheorem{defn}{Definition}
\newtheorem{conj}{Conjecture}
\newtheorem{remark}{Remark}

% adding section to equations numbering
\numberwithin{equation}{section}

% usefull shortcuts
\providecommand{\floor}[1]{\left \lfloor #1 \right \rfloor }
\newcommand{\Z}{{\mathbb Z}} %cph changed from \mathbf
\newcommand{\Q}{{\mathbb Q}}
\newcommand{\R}{{\mathbb R}}
\newcommand{\C}{{\mathbb C}}
\newcommand{\N}{{\mathbb N}}
\newcommand{\FF}{{\mathbb F}}
\newcommand{\fe}{\overline{\mathbb F}}
\newcommand{\fq}{\mathbb{F}_q}
\newcommand{\feq}{\overline{\mathbb F}_q}
\newcommand{\rmk}[1]{\footnote{{\bf Comment:} #1}}
\renewcommand{\mod}{\;\operatorname{mod}}
\newcommand{\ord}{\operatorname{ord}}
\newcommand{\TT}{\mathbb{T}}
\renewcommand{\i}{{\mathrm{i}}}
\renewcommand{\d}{{\mathrm{d}}}
\newcommand{\HH}{\mathbb H}
\newcommand{\Vol}{\operatorname{vol}}
\newcommand{\area}{\operatorname{area}}
\newcommand{\tr}{\operatorname{tr}}
\newcommand{\norm}{\mathcal N} % norm =(\frac{ n+\sqrt{n^2-4}} 2)^2
\newcommand{\intinf}{\int_{-\infty}^\infty}
\newcommand{\ave}[1]{\left\langle#1\right\rangle} %  average
\newcommand{\Var}{\operatorname{Var}}
\newcommand{\Prob}{\operatorname{Prob}}
\newcommand{\sym}{\operatorname{Sym}}
\newcommand{\disc}{\operatorname{disc}}
\newcommand{\CA}{{\mathcal C}_A}
\newcommand{\cond}{\operatorname{cond}} % conductor
\newcommand{\lcm}{\operatorname{lcm}}
\newcommand{\Kl}{\operatorname{Kl}} %Kloosterman sum
\newcommand{\leg}[2]{\left( \frac{#1}{#2} \right)}  % Legendre symbol
\newcommand{\sumstar}{\sideset \and^{*} \to \sum}
\newcommand{\LL}{\mathcal L} %L-function of u
\newcommand{\sumf}{\sum^\flat}
\newcommand{\Hgev}{\mathcal H_{2g+2,q}}
\newcommand{\USp}{\operatorname{USp}}
\newcommand{\conv}{*}
\newcommand{\dist} {\operatorname{dist}}
\newcommand{\CF}{c_0} % Fejer constant
\newcommand{\kerp}{\mathcal K}
\newcommand{\fs}{\mathfrak S}
\newcommand{\rest}{\operatorname{Res}} % resultant
\newcommand{\af}{\mathbb A} % affine line
\newcommand{\Ht}{\operatorname{Ht}}
\newcommand{\Set}{\mathcal P} % set of primes with f(p) square-free

\title[Square-free values of polynomials at primes]{Square-free values of polynomials evaluated at primes over a function field}
\author{Guy Lando}
\address{Raymond and Beverly Sackler School of Mathematical Sciences,
Tel Aviv University, Tel Aviv 69978, Israel}
\email{glando1991@gmail.com}
\date{\today}

\thanks{This work is part of the author’s M.Sc. thesis, written under the supervision of Zeev Rudnick at Tel Aviv University.
 Partially supported by the Israel Science
 Foundation (grant No. 1083/10).}

\begin{abstract}
We study a function field version of a classical problem concerning square-free values of polynomials evaluated at primes. We show that for a square-free polynomial $f\in \fq[t][x]$, there is a limiting density as $n\to \infty$ of primes $P \in \fq[t]$ of degree $n$ such that $f(P)$ is square-free.

Over the integers the analogous result is only known when all irreducible factors of $f$ have degree at most 3.
\end{abstract}

% in amsart documentclass, the abstract comes before maketitle
\maketitle

% create table of contents since the paper is long
\tableofcontents
%\newpage

\section{Introduction}

\subsection{Statement of results}

In this paper we study a function field version of a classical
problem concerning square-free values of polynomials evaluated at
primes. Given a polynomial $f\in \Z[x]$ with integer
coefficients, it is conjectured that there are infinitely many
primes $p$ for which  $f(p)$ is square-free, provided that $f$ has no repeated factor 
and that it satisfies some obvious congruence condition. 
This conjecture is only known to be true for polynomials having all their irreducible factors of degree at most three. 
Moreover, it is believed that
the set $\Set_{f,2}$ of primes $p$ for which $f(p)$ is square-free
has positive density, namely if
\begin{equation}
\Set_{f,2}(x) = \{p\leq x \textnormal{ prime } : f(p) \textnormal{  is square-free}\}
\end{equation}
then
\begin{conj}\label{conj square-free}
Let $f\in\Z[x]$ be a polynomial with no repeated factor.
 Assume that for each prime $p$
there is at least one integer $n_p$ for which $f(n_p)$ is not
divisible by $p^2$. 
Denote $\rho_f(d)$ to be the number of solutions of $f(x)\equiv 0\pmod{d}$ in
invertible residues modulo $d$.
Then
\begin{equation}
|\Set_{f,2}(x)| \sim c_{f,2} \pi(x), \quad x\to \infty,
\end{equation}
where  $\pi(x)$ denotes the number of prime integers not larger than $x$, and the positive density $c_{f,2}$ is given by
\begin{equation}
c_{f,2} = \prod_p \left(1-\frac{\rho_f\left(p^2\right)}{p^2-p}\right),
\end{equation}
\end{conj}

More generally, one can ask for the  density of the set $\Set_{f,k}$
of primes $p$ for which $f(p)$ is $k$-free (meaning $f(p)$ is not
divisible by a $k$-th power). The conjectured density is
\begin{equation}
c_{f,k}=\prod_{P}\left(1-\frac{ \rho_f(p^k)}{\phi(p^k)}\right).
\end{equation}

 Uchiyama \cite{Uchiyama} proved this conjectured density for $k=\deg f$ by a
method that also handles the case $k>\deg f$. The case $k=\deg f-1$
was singled out by Erd\"os, who  conjectured that the set contains
infinitely many primes, and following the works of Hooley
\cite{Hooley},  Nair \cite{Nair1}, \cite{Nair2}  Heath-Brown
\cite{Heath Brown}, Helfgott \cite{Helfgott},  Browning
\cite{Browning} and  Reuss \cite{Reuss}  the quantitative conjecture
for  $k=\deg f-1$ is completely solved. The square-free case ($k=2$)
is currently open for $\deg f>3$. Lee and Murty \cite{Lee and Murty} prove that the ABC conjecture implies the $k$-free conjecture
on primes, for $k \geq 3$. Pasten \cite{Pasten} showed that conjecture~\ref{conj square-free}
follows from the ABC conjecture for number fields.

We turn to the function field version. Let $\fq$ be a finite field
of $q$ elements, where $q=p^m$ is a prime power, and $\fq[t]$ the
polynomial ring. We denote by $M_n(q)$ the set of monic polynomials of
degree $n$. We define the absolute value of $a\in\fq[t]$ to be $|a|=q^{\deg a}$. We denote by $\pi_q(n)$ the set of monic irreducible polynomials of
degree $n$, so that $|\pi_q(n)| =\frac{q^n}{n} + O\left(\frac{q^{n/2}}{n}\right)$.

Let $f(x)\in \fq[t][x]$. Monic irreducible polynomials will be called prime polynomials. A polynomial $a(t)\in\fq[t]$ is called square-free if there is
no $P\in\fq[t]$ such that $\deg P>0$ and $P^2\mid a$. We denote by $\Set(n)=\Set_{f,2}(n)$ the set
of prime polynomials $P(t)\in M_n(q)$ such that $f(P)$ is square-free. 
We prove an analogue of
Conjecture \ref{conj square-free} for square-free values,
also establishing asymptotic bounds on the error term.

A polynomial $f\in\fq[t][x]$ is called square-free if there is
no $P\in\fq[t][x]$ such that $P^2\mid f$ and the degree of $P$ as a polynomial in $t, x$ is positive.

 \begin{thm}\label{main thm}
Assume $f\in \fq[t][x]$ is square-free. For a polynomial $D\in
\fq[t]$, define 
$$\rho_f(D)=|\{C\in\fq[t]: \deg C < \deg D, \gcd(D,C)=1,  f(C)\equiv0\pmod{D}\}|.$$
Then
\begin{equation}\label{Main Eq}
\frac{|\Set_{f,2}(n)|}{|\pi_q(n)|} = c_{f,2} + O_{f,q}\left(\frac
{1}{\log_q n}\right)\quad \mbox{as } n\to \infty,
\end{equation}
with
\begin{equation}
c_{f,2}=\prod_P \left(1-\frac{\rho_f\left(P^2\right)}{|P|^2-|P|}\right),
\end{equation}
 where the product runs over the prime polynomials $P$.
The implied constant in the error term $O_{f,q}\left(\frac
{1}{\log_q n}\right)$ depends only on $f$ and the finite field size $q$.
\end{thm}
Note that the constant $c_{f,2}$ is positive if and only if for all primes $P$, there is some $C\in\fq[t]$ with $\deg C<\deg P^2$, $C$ coprime to $P$, such that $f(C) \neq 0 \pmod{P^2}$. See Section \S~\ref{degenerate case} for a discussion.

\subsection{Plan of the proof}

Take $M\in\N$ which will be chosen later,
 and let
\begin{equation}
\Set'(n, M) =\{a\in \pi_q(n): P^2\nmid f(a), \forall P \mbox{ prime with }
\deg P<M\}
\end{equation}
and
\begin{equation}
\Set''(n, M) =\{a\in \pi_q(n): \exists P, \deg P\geq M,\mbox{ s.t. } P^2\mid
f(a)\}
\end{equation}
Then clearly
\begin{equation}
\mathcal \Set(n) \subset \mathcal \Set'(n, M) \subset \mathcal \Set(n) \cup
\mathcal \Set''(n, M),
\end{equation}
so that
\begin{equation}
|\mathcal \Set'(n, M)|-|\mathcal \Set''(n, M)| \leq  |\mathcal \Set(n)|\leq
|\mathcal \Set'(n, M)|.
\end{equation}
Thus it suffices to give an asymptotic estimate for $|\Set'(n, M)|$ (the ``main
term''), which is easy if $M$ is small, and an upper bound for
$|\Set''(n, M)|$ (the ``error term"). We will show in Proposition \ref{prop:1.5} that
\begin{equation}\label{final for N'}
|\Set'(n, M)| = c_{f,2} \frac{q^n}{n} + O\left(\frac{q^n}{nMq^M}\right)+O\left(\frac{q^{\frac{n}{2}+4q^M+M}}{n}\right),
\end{equation}
and in Proposition \ref{prop:2} that
\begin{equation}\label{final estimate for N''}
|\Set''(n, M)|= O\left(\frac{q^n}{Mq^M}+q^{n\frac{p-1}{p}}\right).
\end{equation}
Choosing $M=\lfloor\log_q\frac{n}{9}\rfloor$ in \eqref {final for N'} and \eqref{final estimate for N''} gives
$$|\Set'(n, M)| = c_{f,2} \frac{q^n}{n} + O\left(\frac{q^n}{n^2\log_q{n}}\right)$$
and
$$|\Set''(n, M)|\ll\frac{q^n}{n\log_q{n}},$$
which together yield
\begin{equation}\label{main result}
|\Set_{f,2}(n)| = c_{f,2} |\pi_q(n)| + O\left(\frac{q^n}{n\log_q n}\right).
\end{equation}
This proves Theorem \ref{main thm}.

The proof of \eqref{final for N'} is carried out using a sieve method and an estimate on the size of the set $\pi_q(n;Q,A)$ of primes in arithmetic progression, defined as:
$$\pi_q(n;Q,A)=\{P\in\pi_q(n) : P\equiv A \pmod{Q}\}.$$

The crucial bound \eqref{final estimate for N''} for the
contribution of large primes uses ideas of Ramsay \cite{Ramsay}
and Poonen \cite{Poonen}, formulated in their work on the
related question of square-free values taken at arbitrary
(non-prime)  polynomials, which we will explain in the proof of \eqref{final estimate for N''}.

As a final comment, we point out that we have dealt here with the
limit of large degree $n$ and fixed finite field size $q$. One can
also ask an analogous question for the limit of $q\to \infty$ and $n$
fixed. Define the content of $f(x)\in \fq[t][x]$ to be the monic greatest common divisor of its coefficients, which is an element of $\fq[t]$. In this case, it follows from the recent work of Rudnick
\cite{Rudnick} that 
for any sequence of finite fields $\fq$ of cardinality $q\to \infty$, and any choice of separable $f_q\in \fq[t][x]$ with square-free content, $f_q(P)$ is square-free with probability 1, that is
\begin{equation}
\lim_{q\to \infty} \frac{|\Set_{f_q,2}(n)|}{|\pi_q(n)|}=1.
\end{equation}

This is because in \cite{Rudnick} it is shown that with probability
1 as $q\to \infty$,  for an arbitrary polynomial $a\in M_n$,
$f_q(a)$ is square-free. Since the primes have positive density
(namely $1/n$) in the set of all monic polynomials of degree $n$,
the result follows.

\section{Estimating the main term}

We now turn back to the proof of Theorem \ref{main thm}. We fixed $M\in\N$ and defined
\begin{equation}
\begin{split}
\nonumber
\Set'(n, M) &=\{a\in \pi_q(n): P^2\nmid f(a), \forall P \mbox{ prime with }\deg P<M\},
\\
\Set''(n, M) &=\{a\in \pi_q(n): \exists P, \deg P\geq M,\mbox{ s.t. } P^2\mid f(a)\}.
\end{split}
\end{equation}
We wish to prove that for $0\ll M\leq\frac{n}{2}$ it holds that
\begin{equation}
\nonumber
|\Set'(n, M)| = c_{f,2} \frac{q^n}{n} + O\left(\frac{q^n}{nMq^M}\right)+O\left(\frac{q^{\frac{n}{2}+4q^M+M}}{n}\right)
\end{equation}
and
\begin{equation}
\nonumber
|\Set''(n, M)|\leq\frac{2q^n\deg f}{Mq^M}+O\left(q^{n\frac{p-1}{p}}\right).
\end{equation}
The rest of the paper will use the following notation:
\begin{itemize}
 \item $n,N$ natural numbers
 \item For $a\in \fq[t]^N $, $a_i$ is the value of the $i$-th coordinate of $a$
  \item $P$ is a prime in $\fq[t]$
  \item $f(x)\in\fq[t][x]$ is a square-free polynomial 
\end{itemize}
For $Q\in\fq[t]$, define $\phi(Q)=|\{a\in\fq[t] : \deg a<\deg Q, \gcd(Q, a)=1\}|$.

In the the following proof, we need an estimate for the size of the set of primes of degree $n$ in an arithmetic progression: $\pi_q(n;Q,A).$
We get the estimate using the Prime Polynomial Theorem in arithmetic progressions, with a remainder term given by the Riemann Hypothesis for curves over a finite field (Weil's Theorem) which was first proved in \cite{Weil}:

\begin{lem}[Weil's Theorem]
 For $Q, A\in\fq[t]$ with $\gcd(Q,A)=1$,
 \begin{equation}\label{eq:4.1}
 |\pi_q(n;Q,A)|=\frac{q^n}{n\phi(Q)}+O\left(\frac{q^{\frac{n}{2}}}{n}\deg Q\right)
 \end{equation}
\end{lem}
 Lemma 1 is proved, similarly to the proof of Theorem 4.8 in \cite{Rosen}, from the Riemann Hypothesis in the same way that the corresponding statement over the integers is deduced from the Generalized Riemann Hypothesis, see e.g. \cite[Chapter 20]{Davenport}. It is done, by using the "Explicit Formula"   to express a sum $\sum_{\substack{\deg(P^k)=n\\p^k=A \mod Q}} \deg P$ over prime powers in the arithmetic progression (the higher prime powers are easily shown to give a negligible amount),  with an average  over zeros of all L-functions associated to Dirichlet characters modulo $Q$; the trivial character gives the main term, and each of the remaining nontrivial characters, for which the associated L-function is a polynomial in $u=q^{-s}$ of degree at most $\deg Q-1$, all of whose inverse zeros are in the disc $|u|\leq \sqrt{q}$, contributes at most $(\deg Q-1)q^{n/2}/\phi(Q)$. The division by $n$ arises from the extra factor of $\deg P$ in the Explicit Formula.

\begin{remark}\label{Discriminant-remark}
Suppose that $f\in\fq[t][x]$ is square-free and $P\in\fq[t]$ is prime. Denote the discriminant of $f$ over $\fq(t)$ by $\Delta(f)$ and denote by $\Delta_{\fq[t]/\langle P\rangle}(f)$ the discriminant of $f$ over $\fq[t]/\langle P\rangle$. It holds that $\Delta(f)\neq0$ and for $P\in\fq[t]$ such that $\deg P>\deg\Delta(f)$, we can conclude that $P\nmid\Delta(f)$.

Now assume that $\Delta_{\fq[t]/\langle P\rangle}(f)=0$ and $P$ does not divide the leading coefficient of $f$. It holds that 
$$\Delta(f)\pmod{P}\equiv\Delta_{\fq[t]/\langle P\rangle}(f)\equiv0\pmod{P}$$
which allows to conclude that $P|\Delta(f)$.
\end{remark}
We will use the following version of Hensel's lemma in the paper
\begin{lem}[Hensel's Lemma]\label{Hensel's Lemma}
 Suppose that $f\in\fq[t][x]$ is square-free and $P\in\fq[t]$ is prime such that $\deg P>\deg\Delta(f)$ where $\Delta(f)$ is the discriminant of $f$ over $\fq(t)$. Also assume that $P$ does not divide the leading coefficient of $f$.
Then
\begin{equation}
\begin{split}
&|\{b\in\fq[t] : \deg b<\deg P^2, f(b)\equiv0\pmod{P^2}\}|
\\
&=|\{a\in\fq[t] : \deg a<\deg P, f(a)\equiv0\pmod{P}\}|
\end{split}
\end{equation}
\end{lem}
\begin{proof}
Suppose that $c\in\fq[t]$ and $f(c)\equiv0\pmod{P}$. Take $d\in\fq[t]$ such that $d\equiv c\pmod P$. Thus $d=c+tP$.
From the formal derivative formula for $f$ we have that
\begin{equation}\label{Hens-lem-eq-1}
f(d)=f(c+tP)=f(c)+tP\frac{\partial f}{\partial x}(c)+P^2(\cdots).
\end{equation}
Now, since $f(c)\equiv0\pmod{P}$ it follows that there is some $s\in\fq[t]$ such that $f(c)=sP$. 
Combining this with \eqref{Hens-lem-eq-1} gives
$$f(d)\equiv (s+t\frac{\partial f}{\partial x}(c))P\pmod{P^2}$$
thus
\begin{equation}\label{Hens-lem-eq-1.25}
f(d)\equiv0\pmod{P^2}\Longleftrightarrow s+t\frac{\partial f}{\partial x}(c)\equiv0\pmod{P}.
\end{equation}
By remark \ref{Discriminant-remark}  we get that
\begin{equation}\label{Hens-lem-eq-1.5}
P\nmid\Delta(f). 
\end{equation}
So if $f(c)\equiv0\pmod{P}$ and $\frac{\partial f}{\partial x}(c)\equiv0\pmod{P}$ then from the basic property of discriminant we get that the discriminant of $f$ over $\fq[t]/\langle P\rangle$, which will be denoted by $\Delta_{\fq[t]/\langle P\rangle}(f)$,
 is 0. By remark \ref{Discriminant-remark} this implies that $P|\Delta(f)$ which is a contradiction to \eqref{Hens-lem-eq-1.5}.
Thus $\frac{\partial f}{\partial x}(c)\not\equiv0\pmod{P}$ and $\frac{\partial f}{\partial x}(c)\pmod{P}$ is invertible. Denote $h(c)=\frac{\partial f}{\partial x}(c)^{-1}\pmod{P}$. Combining this with \eqref{Hens-lem-eq-1.25} we get that 
$$f(d)\equiv0\pmod{P^2}$$ 
for 
\begin{equation}\label{Hens-lem-eq-2}
d\equiv c-f(c)h(c)\pmod{P^2}.
\end{equation}
This gives a solution $d\in\fq[t]$ to $f(d)\equiv0\pmod{P^2}$ such that $d\equiv c\pmod{P}$ and according to the equation \eqref{Hens-lem-eq-2} the solution $d$ is unique modulo $P^2$, which proves the lemma.
\end{proof}

We will also use the following lemma
\begin{lem}\label{bound-lem}
Suppose that $f\in\fq[t][x]$ is square-free and $P\in\fq[t]$ is prime.
Denote the discriminant of $f$ over $\fq(t)$ by $\Delta(f)$, denote by $w_f(t)\in\fq[t]$ the leading coefficient of $f$ as a polynomial in $x$ over $\fq[t]$, and denote by $\deg f$ the degree of $f$ as a polynomial in $x$ over $\fq[t]$.
Then
\begin{equation}
\begin{split}
&|\{b\in\fq[t] : \deg b<\deg P^2, f(b)\equiv0\pmod{P^2}\}|
\\
&\leq \max\{\deg f, q^{2\max\{\deg\Delta(f), \deg w_f\}}\}=O(1)
\end{split}
\end{equation}
and the implied constant in the bound depends only on $f$ and the finite field size $q$.
\end{lem}
\begin{proof}
If $\deg P>\max\{\deg\Delta(f), \deg w_f\}$ then using Hensel's Lemma (lemma \ref{Hensel's Lemma}) we get
\begin{equation}\label{bound-lem-eq-2}
\begin{split}
&|\{b\in\fq[t] : \deg b<\deg P^2, f(b)\equiv0\pmod{P^2}\}|
\\
&=|\{a\in\fq[t] : \deg a<\deg P, f(a)\equiv0\pmod{P}\}|
\end{split}
\end{equation}
and since $\deg P>\deg w_f$ it follows that $f\not\equiv0\pmod{P}$. The number of roots of a polynomial in $x$ over the field $\fq[t]/\langle P\rangle$ is bounded by its degree in $x$ and thus
$$|\{a\in\fq[t] : \deg a<\deg P, f(a)\equiv0\pmod{P}\}|\leq \deg f.$$
Combining this with \eqref{bound-lem-eq-2} we get that if $\deg P>\max\{\deg\Delta(f), \deg w_f\}$ then
\begin{equation}\label{bound-lem-eq-3}
|\{b\in\fq[t] : \deg b<\deg P^2, f(b)\equiv0\pmod{P^2}\}|\leq\deg f.
\end{equation}
On the other hand, if $\deg P\leq\max\{\deg\Delta(f), \deg w_f\}$ then
\begin{equation}\label{bound-lem-eq-4}
\begin{split}
|\{b\in\fq[t] : \deg b<\deg P^2, f(b)\equiv0\pmod{P^2}\}|\leq|P|^2
\\
=q^{2\deg P}\leq q^{2\max\{\deg\Delta(f), \deg w_f\}}.
\end{split}
\end{equation}
The result is obtained by combining \eqref{bound-lem-eq-3} and \eqref{bound-lem-eq-4}.
\end{proof}

Finally, we will use the following lemma which is called the "Explicit Formula" and is proved in Proposition 2.1 in \cite{Rosen}
\begin{lem}[The Explicit Formula]\label{Explicit Formula}
For integers $q$ and $n$, the following holds $\sum_{d\mid n}d|\pi_q(d)|=q^n$. In particular, $i|\pi_q(i)|\leq q^i$ for all integers $i\leq n$.
\end{lem}

We will now prove:
\begin{prop}\label{prop:1.5}
For $0\ll M\leq n$,
$$|\Set'(n, M)| = c_{f,2} \frac{q^n}{n}+O\left(\frac{q^n}{nMq^M}\right)+O\left(\frac{q^{\frac{n}{2}+4q^M+M}}{n}\right).$$
\end{prop}

The proof will use a standard sieve argument.  

\begin{proof}[Proof of Proposition \ref{prop:1.5}]
Denote $w=|\{P:\deg P<M\}|$ and enumerate $\{P:\deg P<M\}=\{P_j:1\leq j\leq w\}$.
Define:
\begin{equation}
\begin{split}
\nonumber
B&=\{(d_1,\dots,d_w) \in (\fq[t])^w:\forall j,1\leq j\leq w, \deg d_j < \deg P_j^2, 
\\
&f(d_j)\not\equiv 0 \pmod{P_j^2}\},
\\
C&=\{(d_1,\dots,d_w)\in B : \forall j,1\leq j\leq w, \gcd(d_j, P_j^2)=1\}.
\end{split}
\end{equation}

Now, for $a\in \pi_q(n)$, if $P_j|d_j$ for some $1\leq j\leq w$, then from $P_j|(a-d_j)$ it follows that $P_j|a$, but both are prime, thus $P_j=a$. Since $\deg a=n$ and $\deg P<M$, it follows that for $n\geq M$, if $P_j|d_j$, then $\{a\in \pi_q(n):\forall 1\leq j\leq w, a\equiv d_j \pmod{P_j^2}\}=\emptyset$.

In addition, according to the Chinese Remainder Theorem, for every set
$(d_1,\dots,d_w)\in B$, there is a unique element $d_{d_1,\dots,d_w}\in\fq[t]$ with $\deg d_{d_1,\dots,d_w} < \deg \prod_{\deg P<M}P^2$ such that $d_{d_1,\dots,d_w}\equiv d_j \pmod{P_j^2}, \forall 1\leq j\leq w$.

Thus for $M\leq n$ it holds that
\begin{equation}\label{eq:4.2}
\begin{split}
&|\Set'(n, M)|=|\{a\in \pi_q(n) : \forall \deg P< M, P^2\nmid f(a)\}|
\\
% one {c}  in array means one column and it is centered
&=\sum_{\begin{array}{c}{(d_1,\dots,d_w)\in B}\end{array}}\left|\left\{a\in \pi_q(n):\forall j,1\leq j\leq w, a\equiv d_j \pmod{P_j^2}\right\}\right|
\\
&=\sum_{\begin{array}{c}{(d_1,\dots,d_w)\in C}\end{array}}\left|\left\{a\in \pi_q(n):\forall j,1\leq j\leq w, a\equiv d_j \pmod{P_j^2}\right\}\right|
\\
&=\sum_{\begin{array}{c}{(d_1,\dots,d_w)\in C}\end{array}}\left|\left\{a\in \pi_q(n):a\equiv d_{d_1,\dots,d_w} \pmod{\prod_{\deg P<M}P^2}\right\}\right|
\end{split}
\end{equation}

Now, using lemma \ref{Explicit Formula} we get
\begin{equation}\label{eq:4.3}
\begin{split}
\deg\prod_{\deg P<M}P^2&=\sum_{\deg P<M}2\deg P=2\sum_{i=1}^Mi|\pi_q(i)|
\\
 \leq2\sum_{i=1}^Mq^i&=2\left(q^M+\frac{q^M-1}{q-1}-1\right)\leq4q^M.
 \end{split}
\end{equation}

Since $\gcd(d_{d_1,\dots,d_w}, \prod_{\deg P<M}P^2)=1$, we can use \eqref{eq:4.1} with \eqref{eq:4.3} to get
\begin{equation}
\begin{split}
\nonumber
&|\{a\in \pi_q(n):a\equiv d_{d_1,\dots,d_w} \pmod{\prod_{\deg P<M}P^2}\}|
\\
&=\frac{q^n}{n\phi(\prod_{\deg P<M}P^2)}+O\left(\frac{q^{\frac{n}{2}}}{n}\deg\prod_{\deg P<M}P^2\right)
\\
&=\frac{q^n}{n\phi(\prod_{\deg P<M}P^2)}+O\left(\frac{q^{\frac{n}{2}}}{n}q^M\right)
\\
&=\frac{q^n}{n\phi(\prod_{\deg P<M}P^2)}+O\left(\frac{q^{\frac{n}{2}+M}}{n}\right).
\end{split}
\end{equation}

 Let us now insert this in equation \eqref{eq:4.2}:

%\begin{align*}\label{eq:4.4}
%\stepcounter{equation}\tag{\theequation}
\begin{equation}\label{eq:4.4}
\begin{split}
&|\Set'(n, M)|=\sum_{\begin{array}{c}{(d_1,\dots,d_w)\in C}\end{array}}\left(\frac{q^n}{n\phi(\prod_{\deg P<M}P^2)}+O\left(\frac{q^{\frac{n}{2}+M}}{n}\right)\right)
\\
&=\left(\prod_{\deg P<M}{\left(\phi\left(P^2\right)-\rho_f\left(P^2\right)\right)}\right)\left(\frac{q^n}{n\phi(\prod_{\deg P<M}P^2)}+O\left(\frac{q^{\frac{n}{2}+M}}{n}\right)\right)
\\
&=\frac{q^n}{n}\prod_{\deg P<M}{\left(1-\frac{\rho_f\left(P^2\right)}{|P|^2-|P|}\right)}+O\left(\frac{q^{\frac{n}{2}+M}}{n}\prod_{\deg P<M}|P|^2\right)
\\
&=\frac{q^n}{n}\prod_{\deg P<M}{\left(1-\frac{\rho_f\left(P^2\right)}{|P|^2-|P|}\right)}+O\left(\frac{q^{\frac{n}{2}+M}}{n}q^{4q^M}\right)
\\
&=\frac{q^n}{n}\prod_{\deg P<M}{\left(1-\frac{\rho_f\left(P^2\right)}{|P|^2-|P|}\right)}+O\left(\frac{q^{\frac{n}{2}+4q^M+M}}{n}\right).
%\end{align*}
\end{split}
\end{equation}

Now, using lemma \ref{bound-lem} we get
\begin{equation}
\begin{split}
\nonumber
\rho_f\left(P^2\right)&=|\{c\in\fq[t] : \deg c<\deg P^2, f(c)\equiv0 \pmod{P^2}, \gcd(c,P^2)=1\}|
\\
&\leq|\{c\in\fq[t] : \deg c<\deg P^2, f(c)\equiv0 \pmod{P^2}\}|=O(1),
\end{split}
\end{equation}
which is a uniform bound for all $P$ and the bound depends only on $f$ and the final field size $q$.
It follows that the infinite product \\ % added to fix overflow
$c_{f,2}=\prod_{P}{\left(1-\frac{\rho_f\left(P^2\right)}{|P|^2-|P|}\right)}$ converges because $\sum_{P}\frac{1}{|P|^2-|P|}$ converges.

By \eqref{eq:4.4},
\begin{equation}\label{eq:4.5}
\begin{split}
|\Set'(n, M)|&=|\{a\in \pi_q(n) : \forall \deg P< M, P^2\nmid f(a)\}|
\\
&=\frac{q^n}{n}\prod_{\deg P<M}{\left(1-\frac{\rho_f\left(P^2\right)}{|P|^2-|P|}\right)}+O\left(\frac{q^{\frac{n}{2}+4q^M+M}}{n}\right)
\\
&=c_{f,2}\frac{q^n}{n}\prod_{\deg P\geq M}\frac{1}{1-\frac{\rho_f\left(P^2\right)}{|P|^2-|P|}}+O\left(\frac{q^{\frac{n}{2}+4q^M+M}}{n}\right).
\end{split}
\end{equation}

Now using the uniform bound $\rho_f\left(P^2\right)=O(1)$, the fact that $\log\frac{1}{1-x} \ll x$ for small enough $x>0$, and lemma \ref{Explicit Formula}, we get

\begin{align*}
\log\left(\prod_{\deg P\geq M}{\frac{1}{1-\frac{\rho_f\left(P^2\right)}{|P|^2-|P|}}}\right)&\ll\sum_{\deg P\geq M}{\frac{1}{|P|^2-|P|}}
\ll \sum_{\deg P\geq M}{\frac{1}{|P|^2}}
\\
&\ll\sum_{i\geq M}{\frac{|\pi_q(i)|}{q^{2i}}}\ll\sum_{i\geq M}{\frac{1}{iq^i}}\ll\frac{1}{Mq^M}.
\end{align*}

Now insert this back in \eqref{eq:4.5} and use the Taylor expansion for $e$ to get
\begin{equation}
\begin{split}
\nonumber
|\Set'(n, M)|&=|\{a\in \pi_q(n) : \forall \deg P< M, P^2\nmid f(a)\}|
\\
&=c_{f,2}\frac{q^n}{n}e^{O\left(\frac{1}{Mq^M}\right)}+O\left(\frac{q^{\frac{n}{2}+4q^M+M}}{n}\right)
\\
&=c_{f,2}\frac{q^n}{n}\left(1+O\left(\frac{1}{Mq^M}\right)\right)+O\left(\frac{q^{\frac{n}{2}+4q^M+M}}{n}\right)
\\
&=c_{f,2}\frac{q^n}{n}+O\left(\frac{q^n}{nMq^M}\right)+O\left(\frac{q^{\frac{n}{2}+4q^M+M}}{n}\right).
\end{split}
\end{equation}
\end{proof}

\section{Bounding the remainder}

I will prove:
\begin{prop}\label{prop:2}
For $0\ll M\leq\frac{n}{2}$:
\begin{equation}
\begin{split}
\nonumber
|\Set''(n, M)|&=|\{a\in \pi_q(n) : \exists P, \deg P\geq M, P^2\mid f(a)\}|
\\
&=O\left(\frac{q^n}{Mq^M}+q^{n\frac{p-1}{p}}\right)
\end{split}
\end{equation}
\end{prop}

In \cite{Ramsay}, Ramsay stated a bound for $|\Set''(n, M)|$. However, his argument only works for the case where $f\in \fq[x]$ has constant coefficients, and even in that case the argument is incomplete. A proof for the general case is given in Poonen \cite{Poonen}. Poonen gave an alternative proof that a multi-variable version of this set has density 0. In his proof he reduces the problem to the calculation of the density of
$$\{a\in\fq[t] : \deg a=n, \exists P, \deg P\geq M, P\mid h(a),g(a)\}$$
for $h(x),g(x)\in\fq[t][x]$ which are coprime as elements of $\fq(t)[x]$. He calculates this density in a more general case (Lemma 5.1 in \cite{Poonen}). We will apply some of his ideas to the special case which is needed for the proof in the present paper.

As it is standard in the problem of counting square-free values of polynomials, we
bound the size of $P''(n,M)$ by separating the cases $\deg P > n/2$ and $n/2 \geq \deg P > M$ (this was
suggested to us by Rudnik). We use Poonen's proof method in \cite{Poonen}. All the results will have explicit error terms.

Proposition \ref{prop:2} will be proven by the following two propositions:

\begin{prop}\label{prop:3}
For $0\ll M\leq\frac{n}{2}$,
$$\left|\left\{a\in M_n(q) : \exists P, \frac{n}{2}\geq\deg P\geq M, P^2\mid f(a)\right\}\right|\ll\frac{q^n}{Mq^M}.$$
\end{prop}

\begin{prop}\label{prop:4}
$$\left|\left\{a\in M_n(q) : \exists P, \deg P>\frac{n}{2}, P^2\mid f(a)\right\}\right|\ll q^{n\frac{p-1}{p}}.$$
\end{prop}

The bound in Proposition \ref{prop:2} is achieved by bounding the sets above which go over non prime polynomials as well as prime polynomials. Those sets, as it turns out, are easier to estimate. In the case of $\Z$ this technique doesn't work for the parallel of Proposition \ref{prop:4} because the bigger set is also hard to bound.

\section{Proof of Proposition \ref{prop:3}}

\begin{proof}

Notice that if $\deg P\leq\frac{n}{2}$, then for any $C\in\fq[t]$ we have
$$|\{a\in M_n(q) : a\equiv C \pmod{P^2}\}|=\frac{q^n}{|P|^2}.$$
Denote $x(D)=|\{C\in\fq[t] : \deg C<\deg D, f(C)\equiv0 \pmod{D}\}|$. Using lemma \ref{bound-lem} and lemma \ref{Explicit Formula} we get

\begin{align*}
&\left|\left\{a\in M_n(q) : \exists P, M\leq\deg P\leq\frac{n}{2}, P^2\mid f(a)\right\}\right|
\\
&\leq\sum_{M\leq\deg P\leq\frac{n}{2}}|\{a\in M_n(q) : P^2\mid f(a)\}|
\\
&=\sum_{M\leq\deg P\leq\frac{n}{2}}\sum_{
% add \tiny here to make font smaller because from some reason it gone bigger
\tiny\begin{array}{c}C\in\fq[t], \deg C <\deg P^2\\f(C)\equiv0\!\!\!\!\pmod{P^2}\end{array}}|\{a\in M_n(q) : a\equiv C\!\!\!\!\pmod{P^2}\}|
\\
&=q^n\sum_{M\leq\deg P\leq\frac{n}{2}}\frac{x\left(P^2\right)}{|P|^2}
\ll q^n\sum_{M\leq\deg P\leq\frac{n}{2}}\frac{1}{|P|^2}
\\
&=q^n\sum_{M\leq k\leq\frac{n}{2}}\frac{|\pi_q(k)|}{q^{2k}}
\leq q^n\sum_{M\leq k\leq\frac{n}{2}}\frac{\frac{q^k}{k}}{q^{2k}}
\\
&=q^n\sum_{M\leq k\leq\frac{n}{2}}\frac{1}{kq^k}
\leq\frac{q^n}{Mq^M}\sum_{M\leq k\leq\frac{n}{2}}\frac{1}{2^k}\leq\frac{q^n}{Mq^M}.
\end{align*}
\end{proof}

Note 1: Poonen in \cite{Poonen} estimated this set (in his paper it is
$|\bigcup_{s=0}^{N-1}Q_s|$) using dimension considerations from
algebraic geometry (see the proof of Lemma 5.1 in \cite{Poonen}).

Note 2: Similar proof for the integer case can be found in \cite{Granville}.

\section{Proof of Proposition \ref{prop:4}}

Denote
\begin{equation}
\begin{split}
\nonumber
&F(y_0,\dots,y_{p-1})=f\left(\sum_{j=0}^{p-1}{t^jy_j^p}\right)\in\fq[t][y_0,\dots,y_{p-1}],
\\
&Q=\left\{a\in \fq[t]^p : \forall 0\leq i<p, \deg a_i\leq\floor{\frac{n}{p}} \textnormal{and } \exists P, \deg P>\frac{n}{2}, P\mid F(a),\frac{\partial F}{\partial t}(a)\right\}.
\end{split}
\end{equation}

Proposition \ref{prop:4} follows from the following two lemmas:

\begin{lem}\label{lem:8}
$$\left|\left\{a\in M_n(q) : \exists P, \deg P>\frac{n}{2}, P^2\mid f(a)\right\}\right|\leq|Q|$$
\end{lem}

\begin{lem}\label{lem:9}
We have that at least one of the following holds:

\begin{itemize}
\item $c_{f,2} = 0$ and $P_{f,2} (n) = 0$ (in which case Theorem 1 holds), or
\item $|Q| \ll q^{n\frac{p-1}{p}}$.
\end{itemize}
\end{lem}

Lemma \ref{lem:8} will be proved using the following lemma:

\begin{lem}\label{lem:7}
Denote by $b$ the unique integer such that $b\equiv n \pmod{p}$ and $0\leq b<p$.

For each $n$ there is a set $A_n\subset (\fq[t])^p$ such that $(1)$
$$M_n(q)=\left\{\sum_{j=0}^{p-1}{t^ja_j(t)^p} : (a_0(t),\dots,a_{p-1}(t))\in A_n\right\}$$
and $(2)$
$$\forall (a_0(t),\dots,a_{p-1}(t))\neq (b_0(t),\dots,b_{p-1}(t))\in A_n, \quad\sum_{j=0}^{p-1}{t^ja_j(t)^p}\neq \sum_{j=0}^{p-1}{t^jb_j(t)^p}$$

Moreover, $a_b(t)\in M_{\floor{\frac{n}{p}}}(q)$ and for all $a \in A_n$ we have that $\deg a_j(t) \leq \lfloor n/p \rfloor$ for each $j$.
\end{lem}

Lemma \ref{lem:9} will be proven using the following proposition

\begin{prop}\label{prop:6}
Let $N\geq0$ be an integer.
For irreducible $f,g\in \fq[t][x_1,\dots,x_N]$ that are coprime in $\fq(t)[x_1,\dots,x_N]$, it holds for $N>0$ that
$$\left|\left\{a\in \fq[t]^N : \deg a_i\leq\floor{\frac{n}{p}}, \exists P, \deg P>\frac{n}{2}, P\mid f(a),g(a)\right\}\right|\ll_{f,g} q^{n\frac{N-1}{p}}.$$
For $N=0$ and $n\gg0$,
$$\left\{P : \deg P>\frac{n}{2}, P\mid f,g\right\}=\emptyset.$$
\end{prop}

Proposition \ref{prop:6} will be proven using the following two propositions

\begin{prop}\label{prop:7}
Let $N\geq1$ be an integer and let $0\neq g\in \fq[t][x_1,\dots,x_{N-1}]$ be a polynomial. Define
$$S_0=\left\{a\in \fq[t]^N : \deg a_i\leq\floor{\frac{n}{p}},g(a)=0\right\},$$
then it holds that
$$|S_0|\ll q^{n\frac{N-1}{p}}.$$
\end{prop}

\begin{prop}\label{prop:8}
Let $N\geq1$ be an integer and let $f\in \fq[t][x_1,\dots,x_{N}]$, $g\in \fq[t][x_1,\dots,x_{N-1}]$ be polynomials. 
Denote by $f_1\in \fq[t][x_1,\dots,x_{N-1}]$ the coefficient of the highest power of $x_N$ in $f$ when looking at $f$ as polynomials in $x_N$.
Define
\begin{equation}
\begin{split}
\nonumber
S&=\{a\in \fq[t]^N : \deg a_i\leq\floor{\frac{n}{p}}, \exists P, \deg P>\frac{n}{2}, P\mid f(a),g(a),
\\
&\left.P\nmid f_1(a),g(a)\neq0\right\},
\end{split}
\end{equation}
then it holds that
$$|S|\ll_{f,g} q^{n\frac{N-1}{p}}.$$
\end{prop}

\begin{proof}[Proof of Lemma \ref{lem:8}]
Denote the sets from Lemma \ref{lem:7} by $A_n$, for each $n$. By Lemma \ref{lem:7} we get
\begin{equation}\label{eq:7.1}
\begin{split}
&\left|\left\{a\in M_n(q) : \exists P, \deg P>\frac{n}{2}, P^2\mid f(a)\right\}\right|
\\
&=\left|\left\{\sum_{j=0}^{p-1}{t^ja_j(t)^p}: (a_0(t),\dots,a_{p-1}(t))\in A_n,\exists P, \deg P>\frac{n}{2},\right.\right.
\\
&\left.\left.P^2\mid f(\sum_{j=0}^{p-1}{t^ja_j(t)^p})\right\}\right|
\\
&=\left|\left\{(a_0(t),\dots,a_{p-1}(t))\in A_n : \exists P, \deg P>\frac{n}{2}, P^2\mid F(a_0(t),\dots,a_{p-1}(t))\right\}\right|
\end{split}
\end{equation}
\begin{equation}
\begin{split}
\nonumber
&=\left|\left\{(a_0(t),\dots,a_{p-1}(t))\in A_n : \exists P, \deg P>\frac{n}{2}, P\mid F(a_0(t),\dots,a_{p-1}(t)),\right.\right.
\\
&\left.\left.P\mid \frac{\d F(a_0(t),\dots,a_{p-1}(t))}{\d t}\right\}\right|.
\end{split}
\end{equation}

Notice that by the total derivative formula for $F$, and since $F(y_0,\dots,y_{p-1})\in\fq[t]\left[y_1^p,\dots,y_{p-1}^p\right]$, it holds that
\begin{align*}\label{eq:7.2}
&\frac{\d F(a_0(t),\dots,a_{p-1}(t))}{\d t}
\\\stepcounter{equation}\tag{\theequation}
&=\frac{\partial F(y_0,\dots,y_{p-1})}{\partial t}(a_0(t),\dots,a_{p-1}(t))
\\
&+\frac{\partial F(y_0,\dots,y_{p-1})}{\partial y_0}(a_0(t),\dots,a_{p-1}(t))\frac{\d a_0(t)}{\d t}
\\
&+\frac{\partial F(y_0,\dots,y_{p-1})}{\partial y_1}(a_0(t),\dots,a_{p-1}(t))\frac{\d a_1(t)}{\d t}+\cdots
\\
&+\frac{\partial F(y_0,\dots,y_{p-1})}{\partial y_{p-1}}(a_0(t),\dots,a_{p-1}(t))\frac{\d a_{p-1}(t)}{\d t}
\\
&=\frac{\partial F(y_0,\dots,y_{p-1})}{\partial t}(a_0(t),\dots,a_{p-1}(t))+0+\cdots+0
\\
&=\frac{\partial F(y_0,\dots,y_{p-1})}{\partial t}(a_0(t),\dots,a_{p-1}(t)).
\end{align*}

Combining \eqref{eq:7.1} and \eqref{eq:7.2} we get:

\begin{equation}\label{eq:7.3}
\begin{split}
&\left|\left\{a\in M_n(q) : \exists P, \deg P>\frac{n}{2}, P^2\mid f(a)\right\}\right|=\\
&\left|\left\{a=(a_0(t),\dots,a_{p-1}(t))\in A_n : \exists P, \deg P>\frac{n}{2}, P\mid F(a),P\mid \frac{\partial F}{\partial t}(a)\right\}\right|.
\end{split}
\end{equation}

Now, by Lemma \ref{lem:7}, for all $a \in A_n$ we have that $\deg a_j(t) \leq \lfloor n/p \rfloor$ for each $j$.
Therefore,
$$\left\{a=(a_0(t),\dots,a_{p-1}(t))\in A_n : \exists P, \deg P>\frac{n}{2}, P\mid F(a),P\mid \frac{\partial F}{\partial t}(a)\right\}\subseteq Q,$$
which together with \eqref{eq:7.3} proves the lemma.
\end{proof}

\begin{proof}[Proof of Lemma \ref{lem:7}]
The existence of such $A_n$ for each $n$ follows from the fact that $\fq[t]$ is a free $\fq[t^p]$ module of rank $p$ with the obvious action.

Denote by $b$ the unique integer such that $b\equiv n \pmod{p}$ and $0\leq b<p$.
By the definition of $A_n$, for all $a \in A_n$ we have that $\deg\sum_{j=0}^{p-1}{t^ja_j(t)^p} = n$. 
Thus, using the fact that degrees of polynomials are integers, we get
\begin{equation}
\begin{split}
\nonumber
&a_b(t)\in M_{\floor{\frac{n}{p}}}(q)\textnormal{ and } \forall j,0\leq j\leq {p-1}, \deg a_j(t)\leq\frac{n-j}{p}\leq\frac{n}{p}\Longrightarrow\\
&\forall j,0\leq j\leq {p-1}, \deg a_j(t)\leq\lfloor n/p \rfloor,
\end{split}
\end{equation}
as desired.
\end{proof}

\begin{proof}[Proof of Lemma \ref{lem:9}]

By Lemma 7.2 in \cite{Poonen}, $F\in\fq[t][y_0,\dots,y_{p-1}]$ is square-free because $f$ is square-free.

The next argument, showing that the case where $F, \frac{\partial F}{\partial t}\in\fq(t)[y_0,\dots,y_{p-1}]$ are not coprime is degenerate, replaces Lemma 7.3  in \cite{Poonen} to allow the proof to be easily generalized to the $k$-free case.

Set $G=\gcd\left(F, \frac{\partial F}{\partial t}\right)\in\fq(t)\left[y_0,\dots,y_{p-1}\right]$. By the total derivative formula one has, just like in the proof of Lemma \ref{lem:8},
that for all $a(t)\in\fq[t]^p$ it holds that $\frac{\d F(a)}{\d t}=\frac{\partial F}{\partial t}(a)$,
because the rest of the partial derivatives vanish since $F\in\fq(t)[y_0^p,\dots,y_{p-1}^p]$.

If $\deg G>0$, then for $a(t)\in\fq[t]^p$ let $P(t)$ be a prime factor of $G(a)$ (such $P(t)$ exists because $\deg G>0$). Then $P^2\mid F(a)$ (because $P\mid G(a)$, $G(a)\mid F(a)$, and $G(a)\mid \frac{\partial F}{\partial t}(a)=\frac{\d F(a)}{\d t}$).

Thus if $\deg G>0$, then for all $a(t)\in\fq[t]^p$ , $F(a)$ is not square-free. Thus by Lemma \ref{lem:7}, for all $a(t)\in\fq[t]$, $f(a)$ is not square-free which gives $|\Set_{f,2}(n)|=0$. Now by Theorem 3.4 in \cite{Poonen} we have
$$\lim_{n\to\infty}\frac{|\{a\in M_n(q) : f(a) \textnormal{ is square-free}\}|}{|M_n(q)|}=$$
$$\prod_P{\left(1-\frac{|\{c\in\fq[t] : \deg c<\deg P^2, f(c)\equiv0 \!\!\!\!\pmod{P^2}\}|}{|P|^2}\right)}.$$
Thus if for some $a(t)\in\fq[t]$, $f(a)$ is not square-free, then
$$\prod_P{\left(1-\frac{|\{c\in\fq[t] : \deg c<\deg P^2, f(c)\equiv0 \!\!\!\!\pmod{P^2}\}|}{|P|^2}\right)}=0.$$
By Lemma \ref{bound-lem} and the fact that the sum $\sum_P{\frac{1}{|P|^2}}$ converges, the infinite product converges. So the vanishing of the product happens only when there is some prime $P$ such that $P^2|f(c)$ for all $c\in\fq[t]$. Thus we get
\begin{equation}
\begin{split}
\nonumber
\rho_f\left(P^2\right)&=\left|\left\{c\in\fq[t] : \deg c<\deg P^2, \gcd(c,P)=1,  f(c)\equiv0\pmod{P^2}\right\}\right|
\\
&=|P|^2-|P|.
\end{split}
\end{equation}
This implies that
$$c_{f,2}=\prod_P \left(1-\frac{\rho_f\left(P^2\right)}{|P|^2-|P|}\right)=0,$$
which concludes the result of the lemma for the case of $\deg G>0$.

Now let us assume that $\deg G=0$, which means that $F$ and $\frac{\partial F}{\partial t}$ are coprime in $\fq(t)[y_0,\dots,y_{p-1}]$.

If the decompositions of $F$ and $\frac{\partial F}{\partial t}$ into irreducibles are $F=f_1\cdots f_{i_f}$ and $\frac{\partial F}{\partial t}=g_1\cdots g_{i_g}$ ($f_i\neq g_j$ because $F$ is square-free, thus $f_i, g_j$ are coprime in $\fq(t)[y_0,\dots,y_{p-1}]$ since they are both irreducible), then
\begin{equation}
\begin{split}
\nonumber
Q&=\left\{a\in \fq[t]^p : \deg a_i\leq\floor{\frac{n}{p}}, \exists P, \deg P>\frac{n}{2}, P\mid F(a),\frac{\partial F}{\partial t}(a)\right\}
\\
&=\bigcup_{i,j}\left\{a\in \fq[t]^p : \deg a_i\leq\floor{\frac{n}{p}}, \exists P, \deg P>\frac{n}{2}, P\mid f_l(a),g_j(a)\right\}.
\end{split}
\end{equation}
Recall that
$$F(y_0,\dots,y_{p-1})=f\left(\sum_{j=0}^{p-1}{t^jy_j^p}\right)\in\fq[t][y_0,\dots,y_{p-1}],$$
thus the number of irreducibles in the decomposition of $F$ is bounded by $p\cdot\deg f$ and thus this also bounds the number of irreducibles in the decomposition of $\frac{\partial F}{\partial t}$.
Now, by proposition \ref{prop:6} we conclude the proof
\begin{equation}
\begin{split}
\nonumber
|Q|&=\left|\bigcup_{i,j}\left\{a\in \fq[t]^p : \deg a_i\leq\floor{\frac{n}{p}}, \exists P, \deg P>\frac{n}{2}, P\mid f_l(a),g_j(a)\right\}\right|
\\
&\ll q^{n\frac{p-1}{p}}\cdot(p\cdot\deg f)^2\ll q^{n\frac{p-1}{p}},
\end{split}
\end{equation}
the product constant depending on the field size $q$ and the polynomial $f$ was neglected because as stated in the main theorem (theorem \ref{main thm}) the error term depends on $f$ and the finite field size $q$. 
\end{proof}

\begin{proof}[Proof of Proposition \ref{prop:6}]
Note that $f, g\neq 0$ because $f$ and $g$ are coprime.

Since we are interested in $P$ with a large value of $|P|$, we may divide $f,g$ by common factors in $\fq[t]$ and assume that $f,g$ are coprime as elements of $\fq[t][x_1,\dots,x_N]$.

Denote:
$$Q'=\left\{a\in \fq[t]^N : \deg a_i\leq\floor{\frac{n}{p}}, \exists P, \deg P>\frac{n}{2}, P\mid f(a),g(a)\right\}.$$
We proceed by induction on $N$.

If $N=0$, then $f,g\in\fq[t]$. Thus for $\frac{n}{2}>\max\{\deg f,\deg g\}$ it holds that $Q'=\emptyset$.

Now assume $N\geq1$. Denote by $f_1,g_1\in \fq[t][x_1,\dots,x_{N-1}]$ the coefficients of the highest power of $x_N$ in $f,g$, respectively, when looking at $f,g$ as polynomials in $x_N$.

Case 1: Assume the $x_N$-degrees of both $f$ and $g$ are positive.
Since $f, g$ are coprime in $\fq[t][x_1,\dots,x_N]$, they are also coprime if viewed as single-variable polynomials in $\fq(t,x_1,\dots,x_{N-1})[x_N]$.
Thus by the B\'ezout Identity, there are $b,c \in \fq(t,x_1,\dots,x_{N-1})[x_N]$ such that
$1=bf+cg$.
Multiplying by the common denominator it follows that there are $B,C \in \fq[t][X_1,\dots,X_{N-1}][X_N]$
and $0\neq D \in \fq[t][X_1,\dots,X_{N-1}]$ such that
$D=Bf+Cg$.

Note: The polynomial $D$ here replaces the resultant used in Poonen's proof of Lemma 5.1 in \cite{Poonen}. The change is done so that it will be easier to generalize the proof to the $k$-free case.

Since $D$ is nonzero and does not involve $x_N$ and since $f,g$ are irreducible, it follows that $D$ is coprime with each of $f,g$.
If $P$ divides $f(a)$ and $g(a)$, then from $D=Bf+Cg$ it follows that $P\mid D(a)$.
Thus,
$$Q'\subseteq\left\{a\in \fq[t]^N : \deg a_i\leq\floor{\frac{n}{p}}, \exists P, \deg P>\frac{n}{2}, P\mid f(a),D(a)\right\}.$$

Since $D$ only depends on $f,g$, its number of factors can be absorbed in the implicit constant implied by $\ll_{f,g}$.
Thus, by looking at the irreducible factors of $D$, just as in the second half of the proof of Lemma \ref{lem:9}, we can assume that $D$ is irreducible. This reduces the problem to dealing with $f,g$ such that one of them is in $\fq[t][x_1,\dots,x_{N-1}]$ and thus does not depend on $x_N$.

Case 2: Suppose that one of $f,g$ is in $\fq[t][x_1,\dots,x_{N-1}]$. Without loss of generality, we can assume that it is $g$. In this case the proof will be by induction on $\delta$, where $\delta$ is the $x_N$-degree of $f$.
If $\delta=0$, then $f, g\in\fq[t][x_1,\dots,x_{N-1}]$ and according to the outer induction hypothesis,
$$\left|\left\{a\in \fq[t]^{N-1} : \deg a_i\leq\floor{\frac{n}{p}}, \exists P, \deg P>\frac{n}{2}, P\mid f(a),g(a)\right\}\right|\ll_{f,g} q^{n\frac{N-2}{p}},$$
whence
\begin{equation}
\begin{split}
\nonumber
&\left|\left\{a\in \fq[t]^{N} : \deg a_i\leq\floor{\frac{n}{p}}, \exists P, \deg P>\frac{n}{2}, P\mid f(a),g(a)\right\}\right|
\\
=&\left|\left\{a\in \fq[t]^{N-1} : \deg a_i\leq\floor{\frac{n}{p}}, \exists P, \deg P>\frac{n}{2}, P\mid f(a),g(a)\right\}\right|
\\
&\cdot\left|\left\{a\in \fq[t] : \deg a\leq\floor{\frac{n}{p}}\right\}\right|
\\
&\ll_{f,g} q^{n\frac{N-2}{p}}\cdot q^{\floor{\frac{n}{p}}}\leq q^{n\frac{N-2}{p}+\frac{n}{p}}=q^{n\frac{N-1}{p}}.
\end{split}
\end{equation}
So let us assume $\delta>0$ and define

\begin{equation}
\begin{split}
\nonumber
S'&=\left\{a\in \fq[t]^N : \deg a_i\leq\floor{\frac{n}{p}}, \exists P, \deg P>\frac{n}{2}, P\mid f_1(a),g(a)\right\},
\\
S_0&=\left\{a\in \fq[t]^N : \deg a_i\leq\floor{\frac{n}{p}},g(a)=0\right\},
\\
S&=\{a\in \fq[t]^N : \deg a_i\leq\floor{\frac{n}{p}}, \exists P, \deg P>\frac{n}{2}, P\mid f(a),g(a),
\\
&\left.P\nmid f_1(a),g(a)\neq0\right\}.
\end{split}
\end{equation}

It holds that $Q'\subseteq S\cup S'\cup S_0$.

If $g|f_1$, then by subtracting a multiple of $g$ from $f$, $Q'$ is not changed and the new $f$ is still coprime to $g$. In this way the degree of $f$ can be lowered, which then allows us to use the inner inductive hypothesis to get the desired result.
So assume now that $g\nmid f_1$ and since we assumed that $g$ is irreducible this means that $g,f_1$ are coprime in $\fq[t][x_1,\dots,x_{N-1}]$. Thus by applying the hypothesis of the outer induction to
$f_1, g\in\fq[t][x_1,\dots_,x_{N-1}]$, just like in the case $\delta=0$, we conclude that 
\begin{equation}\label{S' estimate}
|S'|\ll_{f,g} q^{n\frac{N-1}{p}}.
\end{equation}

Using this together with propositions \ref{prop:7} and \ref{prop:8} allows us to conclude that
$$|Q'|=|S\cup S'\cup S_0|\ll_{f,g} q^{n\frac{N-1}{p}}$$
as desired.
\end{proof}

\begin{proof}[Proof of Proposition \ref{prop:7}]
We will show that $|S_0|\ll q^{n\frac{N-1}{p}}$ by using induction on $N$.

For $N=1$, since $g\in\fq[t][x_1,\dots,x_{N-1}]$, it follows that $g\in\fq[t]$ and since $g\neq 0$,
$$\left\{a\in \fq[t] : \deg a\leq\floor{\frac{n}{p}},g(a)=0\right\}=\emptyset.$$
Let us assume now that the assertion is true for $N-1$, where $N>1$. Denote by $g_2$ the coefficient of the highest power of $x_{N-1}$ in $g$. If $g_2=0$, then $g\in \fq[t][x_1,\dots,x_{N-2}]$. By the induction hypothesis, for all $0\neq h\in \fq[t][x_1,\dots,x_{N-2}]$,
\begin{equation}
\begin{split}
\nonumber
&\left|\left\{a\in \fq[t]^N : \deg a_i\leq\floor{\frac{n}{p}},h(a)=0\right\}\right|=
\\
&\left|\left\{a\in \fq[t]^{N-1} : \deg a_i\leq\floor{\frac{n}{p}},h(a)=0\right\}\right|
\\
&\cdot\left|\left\{a\in \fq[t] : \deg a\leq\floor{\frac{n}{p}}\right\}\right|
\\
&\ll q^{n\frac{N-2}{p}}\cdot q^{\floor{\frac{n}{p}}}\leq q^{n\frac{N-2}{p}+\frac{n}{p}}=q^{n\frac{N-1}{p}},
\end{split}
\end{equation}
as desired. 

Assume now that $g_2\neq 0$. Since $g_2\in \fq[t][x_1,\dots,x_{N-2}]$, it follows as we just seen that
$$\left|\left\{a\in \fq[t]^N : \deg a_i\leq\floor{\frac{n}{p}},g_2(a)=0\right\}\right|\ll q^{n\frac{N-1}{p}}.$$

Now

\begin{equation}
\begin{split}
\nonumber
&\left\{a\in \fq[t]^N : \deg a_i\leq\floor{\frac{n}{p}},g(a)=0\right\}
\\
\subset &\left\{a\in \fq[t]^{N} : \deg a_i\leq\floor{\frac{n}{p}},g_2(a)=0\right\}
\\
&\cup \left\{a\in \fq[t]^N : \deg a_i\leq\floor{\frac{n}{p}},g(a)=0, g_2(a)\neq0\right\}.
\end{split}
\end{equation}

Thus we only need to show that
$$\left|\left\{a\in \fq[t]^N : \deg a_i\leq\floor{\frac{n}{p}},g(a)=0, g_2(a)\neq0\right\}\right|\ll q^{n\frac{N-1}{p}}.$$

Denote by $\deg_{x_{N-1}}g$ the degree of $g$ as a polynomial in the variable $x_{N-1}$. For each $(a_1,\dots,a_{N-2})\in \fq[t]^{N-2}$ there are at most $\deg_{x_{N-1}}g$ values of $a_{N-1}\in\fq[t]$ such that $g(a_1,\dots,a_{N-2}, a_{N-1})=0$; indeed, this holds because $g(a_1,\dots,a_{N-2}, x_{N-1})$ is a polynomial in $x_{N-1}$ of degree $\deg_{x_{N-1}}g$.
Using this and the induction hypothesis gives
\begin{equation}
\begin{split}
\nonumber
&\left|\left\{a\in \fq[t]^N : \deg a_i\leq\floor{\frac{n}{p}},g(a)=0, g_2(a)\neq0\right\}\right|
\\
=&\left|\left\{a\in \fq[t]^{N-1} : \deg a_i\leq\floor{\frac{n}{p}},g(a)=0, g_2(a)\neq0\right\}\right|
\\
&\cdot\left|\left\{a\in \fq[t] : \deg a\leq\floor{\frac{n}{p}}\right\}\right|
\end{split}
\end{equation}
\begin{equation}
\begin{split}
\nonumber
=&\left|\left\{a\in \fq[t]^{N-2} : \deg a_i\leq\floor{\frac{n}{p}},g(a)=0, g_2(a)\neq0\right\}\right|
\\
&\cdot\left|\left\{a\in \fq[t] : \deg a\leq\floor{\frac{n}{p}}\right\}\right|^2 \cdot \deg_{x_{N-1}}g
\\
\ll &q^{n\frac{N-3}{p}}\cdot q^{2\floor{\frac{n}{p}}}\leq q^{n\frac{N-3}{p}+\frac{2n}{p}}=q^{n\frac{N-1}{p}},
\end{split}
\end{equation}

as desired.

\end{proof}

\begin{proof}[Proof of Proposition \ref{prop:8}]
Take $a'=(a_1,\dots,a_{N-1})\in \fq[t]^{N-1}$ with $\deg a_i\leq\floor{\frac{n}{p}}$. We will count the number of $a_N\in\fq[t]$ such that  $a=(a',a_N)\in S$.
Let $\deg g$ be the total degree of $g$. Since $a\in S$, the definition of $S$ implies that $g(a)\neq0$.
Since $g\in\fq[t][x_1,\dots,x_{N-1}]$ does not depend on $x_N$, we get $g(a)=g(a')$.
Therefore,
$$\deg g(a)=\deg g(a')\leq\deg g\cdot\max\{\deg a_1,\dots,\deg a_{N-1}\}\leq\deg g\cdot\floor{\frac{n}{p}}.$$
Thus, since $\floor{\frac{n}{p}}\leq\frac{n}{2}$ and $g(a)\neq0$, there can be at most $\deg g$ different primes $P$ such that $\deg P>\frac{n}{2}$ and $P\mid g(a)$.
Now for $P$ such that $P\nmid f_1(a)$, it holds that $f(a',x_N) \pmod{P}\in(\fq[t]/\langle P \rangle)[x_N ]$ is a polynomial of degree $\delta>0$ over the field $\fq[t]/\langle P\rangle$ (it is a field since $P$ is prime). Thus in this case $f(a',x_N) \pmod{P}$ has at most $\delta$ roots over the field $\fq[t]/\langle P\rangle$ (this is the reason why the case $P\mid f_1(a)$ needs to be dealt with separately, because in that case it might happen that $f(a',x_N) \pmod{P}$ is 0, which would prevent us from bounding the number of its roots).
Now for $P$ such that $\deg P>\frac{n}{2}$ it holds that each $c\in\fq[t]/\langle P\rangle$ has at most one $a_N\in\fq[t]$ such that $\deg a_N\leq\floor{\frac{n}{p}}\leq\frac{n}{2}$ and $a_N\equiv c \pmod{P}$.

We see that there are at most $\deg g$ primes $P$ such that $P\mid g(a')$, and for every such $P$ there are at most $\delta\cdot O(1)$ values of $a_N\in\fq[t]$ with $\deg a_N\leq \frac{n}{2}$ such that $P\mid f(a',a_N)$.

We conclude that for each $a'=(a_1,\dots,a_{N-1})\in \fq[t]^{N-1}$ with $\deg a_i\leq\floor{\frac{n}{p}}$ there are $O(1)$ values of $a=(a',a_N)\in \fq[t]^N$ with $\deg a_N\leq\floor{\frac{n}{p}}$ such that $a\in S$.
Thus
\begin{equation}
\begin{split}
\nonumber
|S|&\ll_{f,g} \left|\left\{(a_1,\dots,a_{N-1})\in \fq[t]^{N-1}:\deg a_i\leq\floor{\frac{n}{p}}\right\}\right|
\\
&\ll_{f,g} q^{\floor{\frac{n}{p}}(N-1)}\ll_{f,g} q^{n\frac{N-1}{p}}.
\end{split}
\end{equation}
\end{proof}

\section{Remarks}

\subsection{Positivity of $c_{f,2}$}\label{degenerate case}

In Theorem \ref{main thm} we proved that
\begin{equation}
\frac{|\Set_{f,2}(n)|}{|\pi_q(n)|} = c_{f,2} + O_{f,q}\left(\frac
{1}{\log_q n}\right)\quad \mbox{as } n\to \infty,
\end{equation}
with
\begin{equation}
c_{f,2}=\prod_P \left(1-\frac{\rho_f\left(P^2\right)}{|P|^2-|P|}\right),
\end{equation}
 where the product runs over the prime polynomials $P$, and for every polynomial $D\in
\fq[t]$, 
$$\rho_f(D)=|\{C\in\fq[t]: \deg C < \deg D, \gcd(D,C)=1,  f(C)\equiv0\pmod{D}\}|.$$
Denote the discriminant of $f$ over $\fq(t)$ by $\Delta(f)$ and denote by $w_f(t)\in\fq[t]$ the leading coefficient of $f$ as a polynomial in $x$ over $\fq[t]$.
We now investigate when $c_{f,2}$ is nonzero:

\begin{prop}\label{prop8}
The following conditions are equivalent:
\begin{enumerate}
\item $c_{f,2}>0$.
\item There are infinitely many primes P such that f(P) is square-free.
\item There is a prime $P\in\fq[t]$ with $\deg P>\max\{\deg\Delta(f),\deg w_f\}$ such that $f(P)$ is square-free.
\item For every prime $P\in\fq[t]$, there is a polynomial $C\in\fq[t]$ such that $P\nmid C$ and $P^2\nmid f(C)$.
\item For each prime $P\in\fq[t]$ such that $\deg P\leq\max\{\deg\Delta(f),\deg w_f\}$, there is a polynomial $C\in\fq[t]$ such that $P\nmid C$ and $P^2\nmid f(C)$. Note that this condition can be checked by a finite computation.
\end{enumerate}
\end{prop}
\begin{proof}
By lemma \ref{bound-lem}, the convergence of the sum $\sum_P{\frac{1}{|P|^2-|P|}}$ implies the convergence of the product in $c_{f,2}$. Consequently, $c_{f,2}=0$ if and only if some term in the product vanishes, which happens if and only if there is some $P$ such that $\rho_f\left(P^2\right)=|P|^2-|P|$. However, from Hensel's Lemma (lemma \ref{Hensel's Lemma}) it follows that for $P$ with $\deg P>\max\{\deg\Delta(f),\deg w_f\}$ we have that $\rho_f\left(P^2\right)=\rho_f(P)\leq |P|-1<|P|^2-|P|$. This proves that $(1)\Leftrightarrow(4)$ and $(5)\Rightarrow(1)$. Also, obviously $(2)\Rightarrow(3)$ and $(4)\Rightarrow(5)$. Now, if $c_{f,2}>0$, then by Theorem \ref{main thm} there is $N$ sufficiently large such that for every $n>N$ there exists a prime $P$ of degree $n$ such that $f(P)$ is square-free. This proves $(1)\Rightarrow(2)$. And finally, if there is some $P\in\fq[t]$ with $\deg P>\max\{\deg\Delta(f),\deg w_f\}$ such that $f(P)$ is square-free, then for all prime $p\in\fq[t]$ with $\deg p\leq\max\{\deg\Delta(f),\deg w_f\}$, we have $p\neq P$, and since they are both prime it follows that $p\nmid P$. But $f(P)$ is square-free, thus we also have $p^2\nmid f(P)$. This proves $(3)\Rightarrow(5)$, which completes the proof of the equivalence of the conditions in the proposition.
\end{proof}

% The 9 determines that the number of digits for reference numbering is 1 (could be any number between 0 to 9)
% one digit is enough since i have only 3 references

\subsection{Final comments}
We make a number of final remarks regarding the proofs in this
paper.
\begin{enumerate}

\item Poonen used a technique that consists of looking only at part of the coordinates of $a\in \fq[t]^{N}$ such that $\exists P, \deg P\geq M, P^2\mid f(a)$, then proving that fixing those coordinates leaves $O(1)$ options for the rest of the coordinates. However, the range of the rest of the coordinates depends on $n$. This proves that the density of the desired set is 0. This technique is not available when working with one variable (last part of the proof of Lemma 5.1 in \cite{Poonen})

\item The reason we need to find $g(x)\in\fq[t][x]$ such
that $P^2\mid f(a)\Longrightarrow P\mid f(a),g(a)$ is in order to use a method to reduce the problem to the case where $g$ depends on one variable less than $f$ does. This allows us, as explained in (1) above, to fix the variables appearing in $g$ and determine the amount of possible values for the variable that does not appear in $g$, but appears in $f$ (last part of the proof of Lemma 5.1 in \cite{Poonen}).

\end{enumerate}

\end{document}